\documentclass[11pt,a4paper]{amsart}
\usepackage{fullpage}
\usepackage{epsfig, amsmath,xspace}
\usepackage{amssymb,amscd,amsthm,mathrsfs}
\usepackage[all,cmtip]{xy}
\usepackage{color}
\usepackage[dvipsnames]{xcolor}
\usepackage[utf8]{inputenc}
\normalfont
\usepackage[T1]{fontenc}
\usepackage[textwidth=14cm,hcentering]{geometry}
\usepackage[colorlinks=true,linkcolor=black,citecolor=blue]{hyperref}

\usepackage{pgfplots}
\usepackage{tikz}
\usepackage{tikz-cd}

\usetikzlibrary{arrows,calc,matrix,topaths,positioning,scopes,shapes,decorations}

\newtheorem{thm}{Theorem}[section]

\newtheorem{proposition}[thm]{Proposition}

\theoremstyle{definition}

\newtheorem{remark}[thm]{Remark}

\newtheorem{defn}[thm]{Definition}
\newtheorem{example}[thm]{Example}

\newdir{ >}{{}*!/-7pt/@{>}}

\def\N{\mathbb{N}}
\def\Z{\mathbb{Z}}

\def\ZZ{\mathcal Z}
\def\BB{\mathcal B}

\def\leq{\leqslant}
\def\geq{\geqslant}
\def\ra{\rightarrow}

\newcommand{\Tot}{{ \rm Tot}}

\newcommand{\Ker}{{\rm Ker}}

\newcommand{\Zzw}{\mathcal{ZW}}

\begin{document}
\title{On the spectral sequence associated to a multicomplex}
\author{Muriel Livernet}
\address[M.~Livernet]{
Univ Paris Diderot, Institut de Math\'ematiques de Jussieu-Paris Rive Gauche, CNRS, Sorbonne Universit\'e, 8 place Aur\'elie Nemours, F-75013, Paris, France}
\email{livernet@math.univ-paris-diderot.fr}

\author{Sarah Whitehouse}
\address[S.~Whitehouse]{School of Mathematics and Statistics\\ 
University of Sheffield\\ S3 7RH\\ England}
\email{s.whitehouse@sheffield.ac.uk}
\author{Stephanie Ziegenhagen}
\email{stephanie.ziegenhagen@web.de}

\begin{abstract}
A multicomplex, also known as a twisted chain complex, has an associated spectral
sequence via a filtration of its total complex. 
We give explicit formulas for all the differentials in this spectral sequence. 
\end{abstract}

\maketitle
\setcounter{tocdepth}{2}

\section{Introduction}
A multicomplex is an algebraic structure generalizing the notion of a (graded) chain complex
and that of a bicomplex. The structure involves a family of higher ``differentials'' indexed by the non-negative integers, and
is also known as a twisted chain complex, or a $D_\infty$-module. 
Multicomplexes have arisen in many different places and play an important role in homotopical and homological algebra.
These objects were first considered by Wall~\cite{Wa61} in his work on resolutions for extensions of 
groups and they were studied by Gughenheim and May~\cite{GM74} in their approach to differential homological algebra.

A multicomplex has an associated total complex, with filtration, and thus an associated spectral sequence. This spectral sequence plays
a key role in the homotopy theory of these objects, as studied in~\cite{CELW18a}. The spectral sequence was studied
by Boardman~\cite{Bo99}, and
by Hurtubise~\cite{Hu10}, who noted that the differentials of the spectral sequence differ from the maps induced by
the higher ``differentials'' of the multicomplex.  The main content of this short note is to give explicit formulas for all the differentials in this spectral sequence. This description generalizes well-known results
in the bicomplex case (see for example~\cite{CFUG97}).

We give some examples, revisiting those given by Hurtubise and Wall, and we briefly note some applications.
In particular, a new application appears in the recent work of Cirici and Wilson~\cite{CW18}.
 They use our description of the $E_2$ page of the spectral sequence, in the case of a multicomplex with only four non-zero structure maps, to introduce and study a new invariant for almost complex manifolds, which generalizes the definition of Dolbeault cohomology for complex manifolds.

\subsection*{Acknowledgements}
We would like to thank Joana Cirici for helpful conversations. We would like to thank the anonymous referee for helping us to clarify our previous version.

\subsection*{Conventions} 
Throughout the paper $k$ will be a commutative unital ground ring.


\section{The spectral sequence associated to a multicomplex}

We begin by introducing multicomplexes, including notation and grading
conventions.

\begin{defn}
\label{def:twchcx}
A \emph{multicomplex} (also called a \emph{twisted chain complex}) is a $(\Z,\Z)$-graded $k$-module $C$ equipped with maps 
$d_i \colon C \ra C$ for $i\geq 0$ of bidegree $\vert d_i \vert = (-i, i-1)$ such that
\[
	\sum_{i+j=n}d_i d_j = 0 \quad \text{for all $n \geq 0.$}
\]
A morphism $f \colon (C, d_i)\ra (C', d_i')$ of multicomplexes is given by maps $f_i \colon C\ra C'$ for $i \geq 0$ of 
bidegree $\vert f_i \vert = (-i,i)$ satisfying
$$\sum_{i+j=n} f_i d_j = \sum_{i+j=n} d'_i f_j \quad \text{for all $n \geq 0$.}$$
\end{defn}

For $C$ a multicomplex and $(a,b)\in\Z\times\Z$, we write $C_{a,b}$ for the $k$-module in bidegree $(a,b)$.

\begin{remark}
Multicomplexes form a category, $\rm{tCh}_k$, with objects and morphisms as in Definition~\ref{def:twchcx}.
Sometimes different sign conventions are adopted. A common alternative is to require the
structure maps to satisfy  the relations
\[
\sum_{i+j=n}(-1)^i d_i d_j = 0 \quad \text{for all $n \geq 0,$}
\]
with a similar sign change for the morphisms. It may be checked that the resulting category is isomorphic
to $\rm{tCh}_k$.

Various other grading conventions may be found, too, such as a single $\N$ or $\Z$ grading, or an $(\N,\Z)$-grading. 
We will discuss where our choice of $(\Z,\Z)$-grading is significant below.
\end{remark}

\begin{remark}
It is shown in ~\cite[10.3.17]{LV12} (singly graded version) and in ~\cite{LRW13} (bigraded version) that multicomplexes
can be viewed as $D_\infty$-algebras, where $D$ is the operad of dual numbers. This point of view is also related to
the work of Lapin~\cite{La01}.
\end{remark}

\begin{example}
If the structure maps of a multicomplex satisfy $d_i = 0$ for $i\geq 1$, we retrieve the notion of a chain complex with an additional grading, sometimes referred to as a vertical bicomplex. If $d_i = 0$ for $i\geq 2$, we retrieve the notion of a bicomplex.
\end{example}

A multicomplex gives rise to a chain complex via totalization. Since we consider $(\Z,\Z)$-gradings it is a priori not clear which is the right notion of total complex in this setting. See~\cite{Me78} for a discussion of this. One could for example associate to a multicomplex $C$ the direct sum total complex with 
$\bigoplus_{a+b =n} C_{a,b}$  in degree $n$. It will turn out that the 
 associated spectral sequence has better convergence properties for the following version of the total complex.

\begin{defn} For a multicomplex $C$, the \emph{associated total complex}  $\Tot C$ is the chain complex with
\[
	(\Tot C)_n = \prod_{\substack{a+b = n\\a\leq 0}} C_{a,b} \oplus \bigoplus_{\substack{a+b=n\\a >0}} C_{a,b} =  \bigoplus_{\substack{a+b = n\\b\leq 0}} C_{a,b} \oplus \prod_{\substack{a+b=n\\b >0}} C_{a,b} .
\]
The differential on $\Tot C$ is given, for $c\in (\Tot C)_n$, by 
\[
(dc)_a =  \sum_{i \geq 0} d_i (c)_{a+i},
\]
where $(c)_a$ denotes the projection of $c$ to $C_{a,*}=\prod_b C_{a,b}$.

Since $(c)_j=0$ for $j$ sufficiently large, the sum above is finite and also $(dc)_a=0$ for sufficiently large $a$, so this formula determines a well-defined map on $\Tot C$.

\end{defn}
\smallskip

Note that it is not possible in general to consider a direct product total complex with  $\prod_{a+b=n} C_{a,b}$ in degree $n$, since in this case the formula above can involve infinite sums.

\medskip

Given a multicomplex $C$, we consider the filtered complex $D$, where
$D:=\Tot C$ filtered by the subcomplexes 
\[
(F_p D)_n =  \prod_{\substack{a+b = n\\a\leq p}} C_{a,b}.
\]

Note that $F_pD=\bigoplus_{i=0}^{r-1} C_{p-i,*}\oplus F_{p-r}D.$
Consequently,  an element $x\in F_pD$ can be written
\begin{equation}\label{decomp}
x=(x)_p+(x)_{p-1}+\ldots +(x)_{p-(r-1)}+u
\end{equation}
 with $u\in F_{p-r}D$, where $(x)_{p-i}$ is the projection of $x$ to $C_{p-i,*}$.

\medskip

We consider the spectral sequence associated to this filtered complex, as presented in~\cite[1.3]{De71}. For $r\geq 0$,
the $r$-stage $E_r(D)$ is an $r$-bigraded complex -- that is, a bigraded module endowed with a square zero map $\delta_r$ of bidegree $(-r,r-1)$ --
and may be written as the quotient
\[E_r^{p,*}(D)\cong \ZZ_r^{p,*}(D)/\BB_r^{p,*}(D),\] where the \textit{$r$-cycles} are given by 
\[\ZZ_r^{p,*}(D):=F_pD\cap d^{-1}(F_{p-r}D)\]
and the \textit{$r$-boundaries} are given by $\BB_0^{p,*}(D)=\ZZ_0^{p-1,*}(D)$ and
\[\BB_r^{p,*}(D):=\ZZ_{r-1}^{p-1,*}(D)+ d\ZZ_{r-1}^{p+r-1,*}(D)\text{ for }r\geq 1.\]

Given an element $x\in \ZZ_r^{p,*}(D)$, we will denote by $[x]_r$ its image in $E_r^{p,*}(D)$. For $[x]_r\in E_r^{p,*}(D)$,
we have 
\begin{equation}\label{EMain}
\delta_r([x]_r)=[dx]_r.
\end{equation}

Expanding the expressions $dx\in F_{p-r}D$ and $dc=x$ for some $c\in F_{p+r-1}D$ using the decomposition (\ref{decomp}) above leads to the following definition.

\begin{defn}\label{P:candb}
Let $x\in C_{p,*}$ and let $r\geq 1$. We define subgraded modules $Z_r^{p,*}$ and $B_r^{p,*}$ of $C_{p,*}$
as follows.
 \begin{align*}
   x\in Z_r^{p,*} \quad \Longleftrightarrow &\quad \text{for } 1\leq j\leq r-1,\; \text{there exists } z_{p-j}\in C_{p-j,*} \text{ such that }  \\
   &d_0x=0 \text{ and } d_nx=\sum_{i=0}^{n-1} d_i z_{p-n+i},  \text{ for all } 1\leq n\leq r-1.  \quad (\star_1) \\
 x\in B_r^{p,*}  \quad  \Longleftrightarrow &\quad \text{for } 0\leq k\leq r-1,\; \text{there exists }  c_{p+k}\in C_{p+k,*} \text{ such that }  \\
 &\quad \begin{cases}
     x=\sum_{k=0}^{r-1} d_kc_{p+k}& \text{and} \\
    0= \sum_{k=l}^{r-1} d_{k-l}c_{p+k} & \text{for }1\leq l\leq r-1. 
    \end{cases}\quad (\star_2) 
\end{align*} 
\end{defn}

\begin{proposition}\label{prop:binz}
For $r\geq 1$ and all $p$, we have $B_r^{p,*}\subseteq Z_r^{p,*}$. 
\end{proposition}

\begin{proof}
Let $x\in B_r^{p,*}$, with  $c_{p+k}\in C_{p+k,*}$ for
$0\leq k\leq r-1$ satisfying  equations~($\star_2$). Define
\[
z_{p-j}=-\sum_{i=0}^{r-1} d_{j+i}c_{p+i} \in C_{p-j,*},
\]
for $1\leq j\leq r-1$. Direct calculation shows that these elements satisfy ($\star_1$) and thus
$ x\in Z_r^{p,*}$.
\end{proof}

\begin{proposition}\label{prop:quotientiso}
The map 
\[\psi:\ZZ^{p,*}_r(D)/\BB^{p,*}_r(D)\rightarrow Z^{p,*}_r/B^{p,*}_r,\]
sending $[x]_r$ to the class 
 $[(x)_p]$,  is well defined and an isomorphism.
\end{proposition}

\begin{proof}
Define
\[\hat\psi: \ZZ_r^{p,*}(D)\rightarrow Z^{p,*}_r/B^{p,*}_r\]
 by $\hat\psi(x)=[(x)_p]$. To see that $(x)_p\in Z^{p,*}_r$, note that $dx\in F_{p-r}D$ implies that $(dx)_{p-n}=0$ for all $0\leq n\leq r-1$. Therefore $d_0(x)_p=0$ and
 \[ d_n(x)_p+\sum_{i=0}^{n-1}d_i(x)_{p-n+i}=(dx)_{p-n}=0,
 \]
 for all $1\leq n\leq r-1$. So, taking $z_{p-n+i}=-x_{p-n+i}$ in Definition \ref{P:candb}, we see $(x)_p\in Z^{p,*}_r$ and a similar argument proves that $\hat\psi$ is surjective.

Let us compute its kernel. Let $x=(x)_p+w \in\Ker\,\hat\psi$, with $w\in F_{p-1}D$. By assumption $(x)_p\in  B^{p,*}_r$, and hence for $0\leq k\leq r-1$ there exists $c_{p+k}\in C_{p+k,*}$ such that
\[ \begin{cases}
     (x)_p=\sum_{k=0}^{r-1} d_kc_{p+k}& \text{and} \\
    0= \sum_{k=l}^{r-1} d_{k-l}c_{p+k}, & \text{for }1\leq l\leq r-1. 
    \end{cases}
    \]
 Let $c=\sum_{k=0}^{r-1} c_{p+k}\in F_{p+r-1}D$. The above relations imply that $(dc)_{p+l}=0$ for all $1\leq l\leq r-1$, and $(dc)_p=(x)_p$. Therefore, $dc\in F_pD$ and $c\in \ZZ_{r-1}^{p+r-1,*}(D)$. In addition, $(x)_p-dc\in 
 F_{p-1}D$, and $x=dc+\rho$, where $\rho=(x)_p-dc+w\in F_{p-1}D$. Then $d^2c=0$ implies that $dx=d\rho\in F_{p-r}D$, and hence 
 $\rho\in\ZZ_{r-1}^{p-1,*}(D)$. Thus $\Ker\,\hat\psi\subseteq   \BB_r^{p,*}(D).$

Conversely if $x\in  \BB_r^{p,*}(D)$, then $x=\rho+dc$ for some  $\rho\in\ZZ_{r-1}^{p-1,*}(D)$ and some $c\in \ZZ_{r-1}^{p+r-1,*}(D)$. So, $\rho\in F_{p-1}D$ and $dc\in F_pD$. Thus, $(x)_p=(dc)_p$
and $(dc)_s=0$ for all $s>p$. This implies that $(x)_p\in B^{p,*}_r$ and $\BB_r^{p,*}(D)\subseteq \Ker\,\hat\psi$.
\end{proof}

\begin{remark}
In the language of \emph{witnesses} adopted in~\cite{CELW18b}, the difference between the 
$\ZZ_r(D)$-cycles and the $Z_r$-cycles is essentially the difference between specifying witnesses and just  
 requiring the existence of them. More precisely, 
$\ZZ^{p,*}_r(D)/F_{p-r}(D)$ corresponds to the
\emph{witness $r$-cycles} for split filtered complexes.
\end{remark}

\begin{thm}\label{th:explicitDiffs}
Under the isomorphism $\psi$ of Proposition~\ref{prop:quotientiso}, the $r$-th differential of the spectral sequence
corresponds to the map $\Delta_r: Z^{p,*}_r/B^{p,*}_r\to Z_r^{p-r,*}/B_r^{p-r,*}$ given by
\[
  \Delta_r([x])= \left[\;d_rx-\sum_{i=1}^{r-1} d_iz_{p-r+i}\;\right],
\]
 where $x \in Z_r^{p,*}$, and the family $\{z_{p-j}\}_{1\leq j\leq r-1}$ satisfies $(\star_1)$.
\end{thm}

\begin{proof} Since  $\{z_{p-j}\}_{1\leq j\leq r-1}$ satisfies $(\star_1)$,  $[x-z_{p-1}-\ldots-z_{p-r+1}]_r\in\ZZ^{p,*}_r(D)/\BB^{p,*}_r(D)$ and
\[\psi[x-z_{p-1}-\ldots-z_{p-r+1}]_r=[x],
\] where $\psi$ is the isomorphism from Proposition~\ref{prop:quotientiso}. Hence
\begin{align*}
\Delta_r([x])&= \psi\delta_r([x-z_{p-1}-\ldots-z_{p-r+1}]_r)\\
&\stackrel{(\ref{EMain})}{=}\psi[d(x-z_{p-1}-\ldots-z_{p-r+1})]_r\\
&=[(d(x-z_{p-1}-\ldots-z_{p-r+1}))_{p-r}]\\
&=[d_rx-\sum_{i=1}^{r-1} d_iz_{p-r+i}].\qedhere
\end{align*}
\end{proof}

\section{Examples}

We revisit the examples given by Hurtubise~\cite{Hu10} in the light of the explicit description of the differentials.
Hurtubise has the same sign and bidegree conventions as ours, but works with ground ring $\Z$.

The first two examples relate to the bicomplex case, that is multicomplexes with $d_i=0$ for $i\geq 2$. The 
first, \cite[Example 1]{Hu10}, is a ``short staircase'' bicomplex, giving a minimal example of non-trivial $\delta_2$ in the spectral sequence in the bicomplex case.
This may be schematically represented as
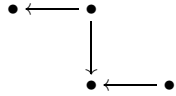
\[
\resizebox{2.5cm}{!}{
\xymatrix{
\bullet&\ar[l]\bullet\ar[d]\\
&\bullet&\ar[l]\bullet
}
}
\]
where each  bullet represents a copy of $\Z$ and  each arrow represents the identity map, the vertical one being a 
$d_0$ and the horizontal ones being $d_1$s.
This bicomplex is (up to minor changes of convention) the bicomplex $\Zzw_2$ of~\cite{CELW18b}, a representing object for the witness $2$-cycles.
The 
second example, \cite[Example 2]{Hu10}, generalizes this to a ``long staircase'' bicomplex, giving a minimal example of non-trivial $\delta_r$ in the spectral sequence in the bicomplex case.  It can be pictured as follows.
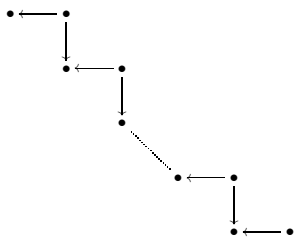
\[
\resizebox{4cm}{!}{
\xymatrix{
\bullet&\ar[l]\ar[d]\bullet\\
&\bullet&\ar[l]\ar[d]\bullet\\
&&\bullet\ar@{.}[dr]&\\
&&&\bullet&\ar[l]\ar[d]\bullet\\
&&&&\bullet&\ar[l]\bullet
}
}
\]
This corresponds to the bicomplex $\Zzw_r$ of~\cite{CELW18b}, a representing object for the witness $r$-cycles.

In~\cite[Example 3]{Hu10}, the first example is modified by putting
in a non-trivial $d_2$, as indicated, with the effect that the $\delta_2$ of the spectral sequence is then zero. 
\[
\resizebox{2.5cm}{!}{
\xymatrix{
\bullet&\ar[l]\bullet\ar[d]\\
&\bullet&\ar[l]\bullet\ar[ull]
}
}
\]

Finally, \cite[Example 4]{Hu10} is indicated below.
\[
\resizebox{7cm}{!}{
\xymatrix{
\bullet\langle d_1z\rangle\ \bullet\langle d_2y\rangle&\ar[l]_(0.35){\footnotesize \begin{pmatrix}1\\0 
\end{pmatrix}}
\bullet\langle z\rangle\ar[d]^<<<{1}\\
&\bullet\langle d_0z=d_1x\rangle&\ar[l]^(0.45){(1,0)}\bullet\langle x\rangle \ \bullet\langle y\rangle\ar[ull]
}
}
\]
Here the diagonal arrow is $d_2$ given by $\begin{pmatrix}0&0\\0&1\end{pmatrix}$. 
Both $x$ and $y$ give rise to elements of $Z_2$, ``witnessed'' by $z$ for $x$ and by $0$ for $y$, and
our formula for $\Delta_2$ gives
\[
\Delta_2([x])=[-d_1z], \qquad \Delta_2([y])=[d_2y].
\]
It is easy to see that $d_1z\not\in B_2$, so $[-d_1z]\not=0$. So we see that the map induced by $d_2$ and the 
second differential in the spectral sequence are both non-zero
and they are different from each other.
\bigskip

We also revisit the original example given by Wall~\cite{Wa61}. Let the group $G$ be an extension of a normal subgroup
$K$ by its quotient group $H$. Wall shows how to construct (inductively) a free resolution of $G$ from free resolutions 
of $K$ and $H$, via what he calls a ``twisted tensor product''. This resolution has the form of $\Tot C$ for $C$ a multicomplex.

The explicit example given by Wall  is for $G$ a split extension of $K=\Z/r$ by $H=\Z/s$, with generators $x,y$, subject to
relations
\[
x^r=y^s=1, \quad y^{-1}xy=x^t, \quad \text{with } t^s\equiv 1 \bmod r.
\]
Applying his construction to the standard resolutions of the cyclic groups, 
he describes a (first quadrant) multicomplex whose $\Tot$ gives a free resolution for $G$.

Tensoring this over $\Z G$ with $\Z$ one obtains the following multicomplex, with homology of its total complex the group homology
of $G$ with integer coefficients. (Note that we switch over the order of Wall's bigradings, so that conventions match the rest
of this paper.)

For $a\geq 0, b\geq0$, $C_{a,b}$ is a free abelian group on generator $c_{a,b}$ and otherwise $C_{a,b}=0$.

Then, for all $a,b$, writing $T_b=\sum_{j=0}^{s-1} t^{jb}$,
\begin{alignat*}{2}
d_0c_{a,2b-1}&=0, \qquad &d_1c_{2a,2b}&=T_bc_{2a-1, 2b},\\
d_0c_{a,2b}&=rc_{a,2b-1}, \qquad &d_1c_{2a,2b-1}&=-T_bc_{2a-1, 2b-1},\\  
d_1c_{2a+1,2b}&=(t^b-1)c_{2a,2b},\qquad &d_2c_{a,2b}&=0,\\
d_1c_{2a+1,2b-1}&=-(t^b-1)c_{2a,2b-1},\qquad &d_2c_{a,2b-1}&=-\frac{t^{bs}-1}{r}c_{a-2, 2b},
\end{alignat*}
and $d_r=0$ for $r>2$.

As Wall notes, the associated spectral sequence degenerates at the $E_2$ term and he computes the group homology
explicitly. From our point of view, we see that, in any bidegree where $Z_2\neq 0$, the formula for $d_2$  is precisely what it has to be in order for
$\Delta_2$ to be zero. 

In more detail, for $x\in Z_2$, we have $\Delta_2([x])=[d_2x-d_1z]$, where $d_0x=0$ and
$z$ is such that $d_1x=d_0z$. 
If $b>0$ is even, then $Z_2^{a,b}=0$ since $d_0$ from this bidegree is multiplication by $r$ which is injective, so we consider
the other cases.

 Suppose $x=\alpha c_{2a-1,2b-1}\in Z_2^{2a-1,2b-1}$. Then $d_0x=0$ and there is
some $z=\beta c_{2a-2, 2b}$ such that  $d_1x=d_0z$. Now,  
\[
d_1x=d_0z \quad\Longleftrightarrow\quad -(t^b-1)\alpha c_{2a-2,2b-1}=r\beta c_{2a-2,2b-1},
\]
so we see that  such a $z$ exists if and only if $r$ divides $(t^b-1)\alpha$ and
then $\beta=-\frac{(t^b-1)\alpha}{r}$. Then 
\[
d_1z=T_b\beta c_{2a-3,2b}=-\frac{t^{bs}-1}{t^b-1}\frac{(t^b-1)\alpha}{r} c_{2a-3,2b}=-\frac{t^{bs}-1}{r}\alpha c_{2a-3,2b}=d_2x,
\]
so that
$\Delta_2([x])=0$. 

Now suppose $x=\alpha c_{2a,2b-1}\in Z_2^{2a,2b-1}$. Then $d_0x=0$ and there is
some $z=\beta c_{2a-1, 2b}$ such that  $d_1x=d_0z$. 
This time, 
\[
d_1x=d_0z \quad\Longleftrightarrow\quad -T_b\alpha c_{2a-1,2b-1}=r\beta c_{2a-1,2b-1},
\]
so we see that  such a $z$ exists if and only if $r$ divides $T_b \alpha $, and then
 $\beta=-\frac{T_b\alpha}{r}$.
Then
\[
d_1z=(t^b-1)\beta c_{2a-2,2b}=-(t^{b}-1)\frac{T_b\alpha}{r} c_{2a-2,2b}=-\frac{t^{bs}-1}{r}\alpha c_{2a-2,2b}=d_2x,
\]
so again $\Delta_2([x])=0$. 

Finally, we consider $x\in Z_2^{a,0}$. We have $Z_2^{2a, 0}=0$ : if $x=\alpha c_{2a,0}\in Z_2^{2a, 0}$ there must be a $z \in C_{2a-1,1}$ such that
$d_1x=s\alpha c_{2a-1,0}=d_0z=0$ and so $x=0$.

So let $x\in Z_2^{2a-1, 0}$ . Then $d_0x=0$ and picking $z=0$, we have 
$d_0z=d_1x=0$. Then $d_1z=0=d_2x$.

Thus we see that  $\Delta_2([x])=[0]$, for every $x\in Z_2$.

\end{document}